\documentclass[10pt]{article}

\usepackage[latin1]{inputenc}
\usepackage{amssymb}
\usepackage{amsmath}
\usepackage{amscd}
\usepackage{amsthm}
\usepackage{amsfonts}
\usepackage{enumerate}
\usepackage{graphicx}
\usepackage{url}
\usepackage{amssymb}
\usepackage[dvips]{color}
\usepackage{epsfig}
\usepackage{mathrsfs}
\usepackage{indentfirst}
\usepackage{subfig}

\usepackage{xy}
\input xy
\xyoption{all}

\DeclareMathOperator{\DES}{DES}
\DeclareMathOperator{\ASC}{ASC}
\DeclareMathOperator{\des}{des}
\DeclareMathOperator{\asc}{asc}
\DeclareMathOperator{\cover}{cover}
\DeclareMathOperator{\exc}{exc}
\DeclareMathOperator{\maj}{maj}

\DeclareMathOperator{\st}{st}

\DeclareMathOperator{\unexc}{unexc}
\DeclareMathOperator{\bdes}{des}
\DeclareMathOperator{\cyclea}{\emph{cycle}}
\DeclareMathOperator{\flip}{\emph{flip}}
\DeclareMathOperator{\amaj}{amaj}
\DeclareMathOperator{\DEX}{DEX}

\newcommand{\majhat}[1]{\widetilde{\maj_{#1}}}
\newcommand{\majtilde}[1]{\majhat{#1}}
\newcommand{\destilde}[1]{\widetilde{\des_{#1}}}
\newcommand{\bdestilde}[1]{{\widetilde{\bdes_{#1}}'}}
\newcommand{\cycle}[1]{\cyclea_{#1}}

\newcommand{\eqdef}{=} 

\renewcommand{\S}[1]{\mathfrak{S}_{#1}}
\newcommand{\C}[1]{\mathfrak{C}_{#1}}
\newcommand{\invcode}{\textsf{Inv-Code}}
\newcommand{\majcode}[1]{\majtilde{#1}\textsf{-Code}}
\newcommand{\majcodeone}{\widetilde{\maj}\textsf{-Code}}
\newcommand{\cut}{\textsf{cut}}
\newcommand{\ins}{\textsf{insert}}

\newcommand{\simga}{\sigma}

\newcommand{\inv}{^{-1}}

\newtheorem{theorem}{Theorem}[section]
\newtheorem{lemma}[theorem]{Lemma}
\newtheorem{proposition}[theorem]{Proposition}

\newtheorem{corollary}[theorem]{Corollary}

\newtheorem{mapping}[theorem]{Mapping}

\theoremstyle{definition}
\newtheorem{example}[theorem]{Example}

\numberwithin{equation}{section}

\renewcommand{\inv}{\text{inv}}

\begin{document}

\begin{center} \begin{LARGE} {\sc \bf An Eulerian permutation statistic and generalizations} \vspace{6pt}

\end{LARGE} { \Large \textsc{Travis Hance \hspace{0.75in} Nan Li}}

\end{center}

\begin{abstract}
Recently, the second author studied an Eulerian statistic (called $\cover$) in the context of convex polytopes, and proved an equal joint distribution of $(\cover,\des)$ with $(\des,\exc)$.
 In this paper, we present several direct bijective proofs that $\cover$ is Eulerian, and examine its generalizations and their Mahonian partners. We also present a quasi-symmetric function proof (suggested by Michelle Wachs) of the above equal joint distribution.
\end{abstract}

\section{Introduction}
\label{intro}

Permutation statistics is well explored subject in mathematics.
MacMahon \cite{MacMahon} considered four different statistics for a
permutation: the number of \emph{descents} ($\des$), the number of
\emph{exceedances} ($\exc$), the number of \emph{inversions}
($\inv$), and the \emph{major index} ($\maj$). For a permutation
$\sigma \in \mathfrak{S}_n$:
\begin{align*}
\DES(\sigma) &\eqdef \{i : \sigma(i) > \sigma(i-1)\}\\
\des(\sigma) &\eqdef |\DES(\sigma)|\\
\maj(\sigma) &\eqdef \sum_{i\in\DES(\sigma)}
\end{align*}
Meanwhile, $\exc$ and $\inv$ are defined:
\begin{align*}
\exc(\sigma) &\eqdef \#\{i : \sigma(i) > i\}\\
\inv(\sigma) &\eqdef \#\{(i,j) : i<j\text{ and }\sigma(i) > \sigma(j)\}
\end{align*}

It is first due to MacMahon, using an algebraic proof, that $\exc$ is equidistributed with $\des$, and that $\inv$ is equidistributed with $\maj$, over $\S n$, i.e.,
$$
\sum_{\sigma\in\S n} x^{\des(\sigma)} = \sum_{\sigma\in\S n}
x^{\exc(\sigma)},\,\, \sum_{\sigma\in\S n} x^{\inv(\sigma)} =
\sum_{\sigma\in\S n} x^{\maj(\sigma)}
$$

 Any permutation
statistic that is equidistributed with $\des$ is said to be
\emph{Eulerian} and a permutation statistic that is equidistributed
with $\inv$ is said to be \emph{Mahonian} (see \cite{Foata2}). We
call a  pair of statistics \emph{Euler-Mahonian} if one is Eulerian
and one Mahonian \cite{Foata2}. In the last thirty years, people
have studied many new statistics, especially those that are Eulerian
and Mahonian, for example \cite{Clarke, Foata3, Foata5, Foata,
Rawlings2}. There are also many connections with other areas, for
example \cite{Simion, Stanley}.

\subsection{Cover is Eulerian}

The second author recently defined a new Eulerian permutation
statistic called ``cover"  motivated by her study of the $h$-polynomial of
the half-open hypersimplices \cite{ln}.

There she defined
a poset on $\S n$ representing the triangulation of a hypersimplex,
and $\cover(\sigma)$ as the number of elements $\sigma$ covers in the poset.
It is then shown that $\cover$ can be computed as

\begin{equation}\label{coverdef}
\cover(\sigma) = \#\{i\in[n-1]\mid
\sigma^{-1}(i)<\sigma^{-1}(i+1)\} + \left\{\begin{tabular}{ll}0 & if $\sigma(1) = 1$\\ 1 & if $\sigma(1)\ne 1$\end{tabular}\right.
\end{equation}

For example, part of the poset for $n=4$ is shown below,
where each arrow represents a cover relation in the poset and points toward the larger permutation; the bullets represent other permutations not displayed.

$$
\xy 0;/r.2pc/: (0,0)*{3142}="a";
 (0,25)*{3214}="b";
 (-10,10)*{4213}="c";
 (10,10)*{2143}="d";
 (-27,4)*{4312}="e";
 (27,4)*{1432}="f";
 (-14,-7)*{3241}="g";
 (14,-7)*{2431}="h";
 (0,-15)*{4132}="i";
 (-16,-25)*{4231}="j";
 (16,-25)*{3421}="k";
 (-28,-11)*{\bullet}="gg";
 (28,-11)*{\bullet}="hh";
 (-30,-27)*{\bullet}="jj";
 (30,-27)*{\bullet}="kk";
 {\ar "c";"b"};%
 {\ar "d";"b"};%
 {\ar "e";"c"};%
 {\ar "a";"c"};%
 {\ar "a";"d"};%
 {\ar "f";"d"};%
 {\ar "g";"e"};%
 {\ar "g";"a"};%
 {\ar "h";"a"};%
 {\ar "h";"f"};%
 {\ar "i";"a"};%
 {\ar "j";"g"};%
 {\ar "j";"i"};%
 {\ar "k";"i"};%
 {\ar "k";"h"};%
 {\ar@{->} "jj";"j"};%
  {\ar@{->} "hh";"h"};%
   {\ar@{->} "kk";"k"};%
    {\ar@{->} "gg";"g"};%
\endxy
$$

We can see that the cover of $\sigma = 3~1~4~2$ is $3$, either by looking at the poset or using definition (\ref{coverdef}), and seeing that we have $\sigma^{-1}(2) > \sigma^{-1}(1) + 1$, $\sigma^{-1}(4) > \sigma^{-1}(3) + 1$, and $\sigma(1) = 3\ne 1$.
Or, if we take $\sigma = 2~4~3~1$, we see that $\cover(\sigma) = 2$, because $\sigma^{-1}(3) > \sigma^{-1}(2) + 1$, and $\sigma(1) = 2 \ne 1$.

We do not go into details about the poset or the hypersimplices here; the definition (\ref{coverdef}) will be enough for this paper. However; using two different shellable triangulations of the half-open
hypersimplices, the following identity can be proved.

\begin{proposition}\cite{ln}\label{ed}
\begin{center}$\displaystyle\sum_{\sigma\in\mathfrak{S}_n}t^{\des(\sigma)}x^{\cover(\sigma)}= \sum_{\sigma\in\mathfrak{S}_n}t^{\exc(\sigma)}x^{\des(\sigma)},$\end{center}
\end{proposition}

From the above result, we can see that $\cover$ is Eulerian. We will
provide several more direct proofs that $\cover$ is Eulerian in the paper. In Section \ref{euleriansection} we will compare $\destilde2$ with $\exc$, and in Section \ref{mahoniansection}, we will use bijections motivated by the Mahonian statistics to show that $\destilde2$ is equidistributed with $\des$.

\subsection{Mahonian partner}

Consider the following variants on the statistic $\des$. The simplest is the number of \emph{ascents}, denoted $\asc$:
\begin{align*}
\ASC(\sigma) &= \{i : \sigma(i) < \sigma(i+1)\}\\
\asc(\sigma) &= |\ASC(\sigma)|
\end{align*}
We also consider \emph{$2$-descents}:
\begin{align*}
\DES_2(\sigma) &= \{i : \sigma(i) \ge \sigma(i+1) + 2\}\\
\des_2(\sigma) &= |\DES_2(\sigma)|
\end{align*}
We also define a major version,
\begin{align*}
\maj_2(\sigma) &= \sum_{i \in \DES_2(\sigma)} i
\end{align*}
Finally, we define the following two variants on $\des_2$ and $\maj_2$:
\begin{align}
\label{des2def}\destilde2(\sigma) &= \left\{\begin{tabular}{ll}$\des_2(\sigma)$ & if $\sigma(1) = n$\\ $\des_2(\sigma) + 1$ & if $\sigma(1) \ne n$\end{tabular}\right.\\
\label{maj2def}\majhat2(\sigma) &= \maj_2(\sigma) + \asc(\sigma^{-1})
\end{align}
It turns out that $\destilde2$ has a very natural relationship with cover. For any permutation $\sigma \in \S n$, define $\sigma' \in \S n$ by $\sigma'(i) = n+1 - \sigma^{-1}(i)$. Then it is easy to see that $\cover(\sigma) = \destilde2 (\sigma)$. Thus, for the rest of the paper we will be considering $\destilde2$ rather than cover, due to its close analog with $\des$. In Section \ref{mahoniansection}, we will show that
\begin{equation}\label{desmajdes2maj2}\displaystyle \sum_{\sigma\in\S n}x^{\des(\sigma)}y^{\maj(\sigma)} = \sum_{\sigma\in\S n}x^{\destilde2(\sigma)}y^{\majhat2(\sigma)}\end{equation}


%

From the above result, we can see that $\majhat2$ is a Mahonian partner
for $\destilde2$. In this paper, we will provide a bijective proof for
that in Section \ref{mahoniansection}. The bijection is mainly the one used in \cite{Rawlings} which shows that the statistic we will call $\majhat k$ is Mahonian, but we will also use techniques from \cite{Skandera}. To define $\majhat k$, we first define $k$-analogs of $\des$: The statistic $\majhat k$ is defined as follows:
\begin{align}
\DES_k(\sigma) &= \{i : \sigma(i) \ge \sigma(i+1) + k\}\\
\des_k(\sigma) &= |\DES_k|\\
\destilde k(\sigma) &= \left\{\begin{tabular}{ll}$\des_k(\sigma)$ & if $\sigma(1) > n + 1 - k$\\ $\des_k(\sigma) + 1$ & otherwise\end{tabular}\right.
\end{align}
and then we define
\begin{align}
\maj_k(\sigma) &= \sum_{i \in \DES_k(\sigma)}i\\
\majhat k (\sigma) &= \maj_k(\sigma) + \{(i, j) : i < j\text{ and }\sigma(i) < \sigma(j) < \sigma(i) + k\}\label{majkdef}
\end{align}
It should be noted that these definitions are consistent with those of $\des$, $\des_2$, $\destilde 2$, $\maj_2$, and $\majhat2$.



We will also prove a generalization of (\ref{desmajdes2maj2}), namely that
\begin{center}
$\displaystyle \sum_{\sigma \in \S n}x^{\des_k(\sigma)} y^{\majhat k(\sigma)} = \sum_{\sigma \in \S n}x^{\destilde{k+1}(\sigma)}y^{\majhat{k+1}(\sigma)}$.
\end{center}

\section{A bijection for $\exc$ and $\destilde2$}
\label{euleriansection}

\subsection{A map taking $\exc$ to $\des$}
\label{excanddes}

In order to show that $\destilde2$ is Eulerian, we will present a bijection which takes the statistic $\exc$ to $\destilde2$. In order to understand the motivation for this bijection, we first present a bijection which takes $\exc$ to $\des$, given in \cite{EC1}, which we will call $\cycle{0}$. It is also known as the Foata map \cite{Foata3}.

For a permutation $\sigma \in \S n$, define the \emph{standard cycle notation} of $\sigma$ to be
\begin{center}
$\sigma = (a_{1,1}~ a_{1,2}~\cdots ~ a_{1, \ell(1)})(a_{2,1}~ a_{2,2}~\cdots ~ a_{2, \ell(2)})~\cdots ~(a_{k,1}~ a_{k,2}~\cdots ~ a_{k, \ell(k)})$
\end{center}
where (i) $\sigma(a_{i,j}) = a_{i,j+1}$ for $1\le i\le k$, and $1\le j \le \ell(i)
 - 1$, (ii) $\sigma(a_{i, \ell(i)}) = a_{i,1}$, (iii) each number in $\{1,...,n\}$ is equal to $a_{i,j}$ for exactly one pair $(i,j)$, and (iv) $a_{1,1} < a_{2,1} < \cdots < a_{k,1}$, but $a_{i,1} > a_{i,j}$ for all $i$ and $j\ge 2$. For example, the standard cycle notation of the permutation $\sigma = 4~5~6~1~2~7~8~3$ is $(4 1)(5 2)(8 3 6 7)$.

Now, given a permutation $\sigma$, write the standard cycle notation of $\sigma^{-1}$, and erase the parentheses. This gives a permutation $\pi$. Thus, if we take $\sigma = 3~4~1~5~2$, the standard cycle notation of $\sigma^{-1}$ is $(3~1)(5~4~2)$, so $\pi = 3~1~5~4~2$.

We can reverse this map as well. Given $\pi$, there is a unique way to split it into cycles such that it satisfies condition (iv) of standard cycle notation. We simply find each element of $\pi$ that is larger than each element before it. Such elements are known as the \emph{left-to-right} maxima of $\pi$ \cite{EC1}. We let the left-to-right maxima be the first elements of the cycles. For instance, if $\pi = 3~4~2~1~5$, then the left-to-right maxima are $3$, $4$, and $5$, so we partition it into $(3)(4~2~1)(5)$, and then we get $\sigma = 2~4~3~1~5$.

It is not hard to see that for any $\sigma \in \S n$, we have $\exc(\sigma) = \des(\pi)$, as in the example.

\subsection{A map between $\exc$ and $\destilde2$}

We want to create a bijection between $\exc$ and $\destilde2$. It will be slightly more convenient to construct the bijection $\exc$ to a slightly different statistic. Define
\begin{center}
$\bdestilde k(\sigma) = \left\{\begin{tabular}{ll}$\des_k(\sigma)$ & if $\sigma(n) < k$\\ $\des_k(\sigma) + 1$ & if $\sigma(n) \ge k$\end{tabular}\right.$
\end{center}
Let $\flip$ be the bijection which takes $\sigma$ to $\sigma'$, defined by $\sigma'(i) = n+1 - \sigma(n+1-i)$. By comparing the definition of $\bdestilde k$ with the definition of $\destilde k$ (see (\ref{des2def})), we see that the involution $\flip$ takes one statistic to the other (i.e., $\destilde k(\sigma) = \bdestilde k(\sigma')$). We will construct a bijection based on the Foata map which takes $\exc$ to $\bdestilde 2$. By composing this bijection with $\flip$, we obtain a bijection that takes $\exc$ to $\destilde2$. Note that while we could have chosen this as the definition of $\destilde k$ to begin with, it is our original definition that will be more natural in Section \ref{mahoniansection}.

Now we describe our bijection, which we will call $\cycle{-1}$. Take any permutation $\sigma\in\S n$. Construct a graph $G$ as follows: let its vertices be the set $\{0,...,n\}$. For each $i\in\{1,...,n\}$, draw a directed edge from $\sigma(i)$ to $i-1$. For instance, the graph $G$ for $\sigma = 7~8~3~5~1~2~4~9~6$ is given in Figure \ref{graphfig}.

\begin{figure}
\begin{center}\includegraphics{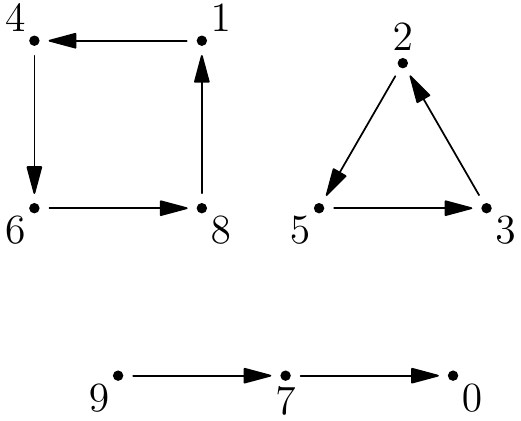}\end{center}
\caption{The graph $G$ for $\sigma = 7~8~3~5~1~2~4~9~6$ in the $\cycle{-1}$ bijection}
\label{graphfig}
\end{figure}

Each vertex of $G$ except $0$ will have exactly one outgoing edge, and each vertex except $n$ will have exactly one incoming edge.
Thus, $G$ consists of some cycles and a path from $n$ to $0$. Define the \emph{standard cycle-path notation} of $G$ as follows: first write the standard cycle notation (as defined in Section \ref{excanddes}) of the cycles, and then write the path from $n$ to $0$. For example, the standard cycle-path notation of $G$ in our example is $(5~3~2)(8~1~4~6)(9~7~0)$. Finally, we let $\pi$ be the permutation obtained by removing the parentheses and removing the trailing $0$. Hence, in this case, $\pi = 5~3~2~8~1~4~6~9~7$. The defines our mapping $\cycle{-1}$. Note that we can reverse this map in the same way we reversed $\cycle{0}$, the foata map; that is, we just partition $\pi$ into the cycles, and one path, by looking at the left-to-right maxima and making these the starts to the cycles and the path. (In particular $n$, will be the first element of the path).

We want to show that if $\cycle{-1}$ takes $\sigma$ to $\pi$, then $\exc(\sigma) = \bdestilde2(\pi)$, thus proving that $\destilde2$ is Eulerian. In fact, we will prove a stronger statement. Define
\begin{center}$\exc_k(\sigma) = \{i : \sigma(i) \ge i + k\}$\end{center}
Then,

\begin{proposition}\label{excdesbij} For $k\ge 1$, if $\cycle{-1}$ maps $\sigma$ to $\pi$, then $\exc_k(\sigma) = \bdestilde{k+1}(\pi)$.\end{proposition}
For $k=1$, we get that $\cycle{-1}$ maps $\exc$ to $\bdestilde2$.
\begin{proof}[Proof of Proposition \ref{excdesbij}]
Suppose we extend $\pi$ to $\hat{\pi}$ such that it has a $0$ at the end (i.e., $\hat{\pi}(n+1) = 0$), as if we never removed $0$ from the end of the standard cycle-path notation. Then, the case where $\hat{\pi}(n) \ge k+1$ becomes a $(k+1)$-descent. We will show that we can find a mapping between the $k$-exceedances of $\sigma$ and the $(k+1)$-descents of this extended $\hat{\pi}$.

Suppose we have a $k$-exceedance $\sigma(i) \ge i + k$. In $G$, $\sigma(i)$ points to $i-1$, which is at most $\sigma(i) - (k+1)\le\sigma(i)$. Now, $i-1$ is not the largest element in its cycle or path, since $\sigma(i)$ is in the same path, so in the standard cycle-path notation, $\sigma(i)$ comes directly before $i-1$, giving a $(k+1)$-descent.

Conversely, suppose we have a $(k+1)$-descent $\hat\pi(i) \ge \hat\pi(i+1) + k+1$. Now, it cannot be the case that $\hat\pi(i+1)$ is at the start of a cycle or path when $\hat\pi$ is decomposed into cycles and a path, since $\hat\pi(i)$ is larger than $\hat\pi(i+1)$. Therefore, $\hat\pi(i)$ and $\hat\pi(i+1)$ are in the same cycle or path, and so $\sigma(\hat\pi(i+1) + 1) = \hat\pi(i)$. (Recall that $G$ is such that $\sigma(j)$ points to $j-1$.) This is a $k$-exceedance.
\end{proof}

We will now show another interesting property of this map. Define the statistic $\unexc$ by
\begin{center}$\unexc(\sigma) = \#\{i : \sigma(i) < i\}$\end{center}
Then,
\begin{proposition}\label{unexcascbij}
If $\cycle{-1}$ maps $\sigma$ to $\pi$, then $\unexc(\sigma) = \asc(\pi)$.
\end{proposition}
\begin{proof}
We will map unexceedances of $\sigma$ to the ascents of $\pi$.

Suppose $\sigma(i) < i$ is an unexceedance. If $\sigma(i) = i-1$, then $\sigma(i)$ points to $\sigma(i)$ in $G$. In standard cycle-path notation, $\sigma(i)$ will be in a cycle by itself, and the first element of the next cycle or path will be larger than it, so this is an ascent. If $\sigma(i) < i-1$, then we have $\sigma(i)$ pointing to $i-1$, which is larger than $\sigma(i)$. Now, $\sigma(i)$ and $i-1$ will be in the same cycle or path. Either $\sigma(i)$ comes directly before $i-1$, in which case we have an ascent immediately; otherwise, $\sigma(i)$ is at the end of some cycle and $i-1$ is at the beginning. But we always have an ascent from one cycle to the next, so even in this case, $\sigma(i)$ is still part of an ascent.

Now we do the reverse direction; suppose $\pi(i) < \pi(i+1)$. We have two cases. If $\pi(i)$ and $\pi(i+1)$ are in the same cycle or path when $\pi$ is partitioned, then we simply have $\sigma(\pi(i+1) + 1) = \pi(i)<\pi(i+1)+1$, which is an unexceedance. Otherwise, $\pi(i)$ is at the end of its cycle, and $\pi(i+1)$ is at the beginning of the next cycle or path. In $G$, we know $\pi(i)$ will point to the element at the beginning of its own cycle, which is at least as large as $\pi(i)$ (possibly the same, if the cycle is of length one). Call the element at the beginning of this cycle $\pi(j)$. Then we will have $\sigma(\pi(j) + 1) = \pi(i)<\pi(j)+1$, which is again an unexceedance.
\end{proof}

Now, not only do we know that $\destilde2(\sigma)$ is Eulerian, but the following corollaries are immediate from the Propositions \ref{excdesbij} and \ref{unexcascbij}:
\begin{corollary}
The multivariate statistic $(\unexc,\exc, \exc_2, \exc_3 ...)$ is equidistributed with the multivariate statistic $(\asc,\destilde2, \destilde3, \destilde4...)$.
\end{corollary}
\begin{proof}The first statistic is taken to the second by applying the mapping $\cycle{-1}$ followed by the mapping of $\flip$. For instance, continuing with the example of $\sigma = 7~8~3~5~1~2~4~9~6$, we apply $\cycle{-1}$ to get $\pi = 5~3~2~8~1~4~6~9~7$ (as stated previously), and then applying $\flip$ we get $\pi' = 3~1~4~6~9~2~8~7~5$.\end{proof}

\begin{corollary}
$(\asc, \destilde2)$ has the same distribution as $(\destilde2, \asc)$
\end{corollary}
\begin{proof}
$(\asc, \destilde2)$ has the same distribution as $(\unexc, \exc)$. By the involution $\flip$, we see that $(\unexc, \exc)$ has the same distribution as $(\exc, \unexc)$. The result follows.
\end{proof}

\section{Mahonian statistics}\label{mahoniansection}

\subsection{Codes}{codes-sec}

In order to study Mahonian statistics such as $\maj$, we will use the notion of a \emph{code}. A code $c$ of length $n$ is a sequence of integers $c(1),\dotsc,c(n)$ such that $0 \le c(i) < i$ for all $i$. We let $\C  n$ denote the set of codes of length $n$; it is clear that there are $n!$ such codes.

Codes, as examined by Skandera \cite{Skandera}, provide a natural way to think about the Mahonian distribution. Define the statistic $\sum$ by
\begin{equation}
\sum(c) = c(1) + \dotsb + c(n)
\end{equation}
for $c \in \C  n$. Then the statistic $\sum$ is Mahonian. In \cite{Skandera}, it is shown that sum over $\C n$ has the same distribution as $\maj$ and $\inv$ over $\S n$.

Before looking at $\maj$ in the next section, let us consider the simpler case of $\inv$, as shown in \cite{Skandera}, to understand how we can use codes to study Mahonian statistics. Consider any permutation $\simga \in \S n$. We construct a code $c \in \C n$ as follows. Simply let
\begin{equation}
c(i) = \#\{j : j > \sigma^{-1}(i), \sigma(j) < i\}.
\end{equation}
Then $\inv(\sigma) = \sum(c)$. We all this mapping an $\inv$-coding scheme, and we denote it by $\invcode$. It is not difficult to see that this map is a bijection; indeed, we can reverse it in this way: to construct $\sigma$ from $c$, start with the permutation$1$. Then, for $i \ge 2$, insert $i$ into the permutation at the $c(i)^\text{th}$ spot from the end (the $0^\text{th}$ position being the end, the $(i-1)^\text{th}$ position being the beginning). For instance, if we have $4~1~2~3$ and $c(5) = 3$, we would insert $5$ to get $4~5~1~2~3$.

Skandera also considers another statistic on codes, $\st$. He defines $\st$ to be the length of the longest sequence $p_1 < p_2 < \dotsb < p_\ell$ such that $c(p_i) > i$ for all $i$. For our purposes it will be more convenient to use the following inductive definition, which is easily seen to be equivalent. Let $\st(0) = 0$, and for $n \ge 2$ and $c \in \C n$, let
\begin{equation}\label{st-def}
\st(c) = \left\{\begin{tabular}{ll}$\st(c_{n-1})$ & if $c(n) \le \st(c_{n-1})$\\ $\st(c_{n-1}) + 1$ & if $c(n) > \st(c_{n-1})$\end{tabular}\right.
\end{equation}
where $c_{n-1} \in \C{n-1}$ is the prefix $c(1),\dotsc,c(n-1)$ of $c$.

Skandera then shows that the distribution of $\st$ over $\C n$ is Eulerian. He does this by constructing a \emph{$\maj k$-coding scheme}, that is, a bijection $\majcodeone : \S n \to \C n$ for which $\maj(\sigma) = \sum(\majcodeone(\sigma))$. The bijection also has the property that $\des(\sigma) = \st(\majcodeone(\sigma))$, and thus has the same distribution over $\C n$ as $\des$ does over $\S n$. We will provide such a bijection in this paper; however, it will be a slightly modified version of Skandera's bijection.

Rawlings \cite{Rawlings} generalizes the bijection. He provides, for any constant $k$, what we will call a \emph{$\majtilde k$-coding scheme}, a bijection $\majcode k : \S n \to \C n$, such that
\begin{equation}
\majtilde k(\sigma) = \sum(\majcode k(\sigma).
\end{equation}
This shows that $\maj k$ is Mahonian for any $k$.

\subsection{Our results} \label{our-results-sec}

We take these ideas a step further, by combining Rawlings' bijections with analysis of the $\st$ statistic. First, we generalize the $\st$ statistic to a more general $\st_k$ (such that $\st_1$ is identical to $\st$). Again, for any $k$, we define $\st_k$ inductively. Let $\st_k(0) = 0$, and for $n\ge 2$ and $c \in \C n$, we let
\begin{equation}\label{stk-def}
\st_k(c) = \left\{\begin{tabular}{ll} $\st_k(c)$ & if $c(n) \le \st k (c_{n-1}) + k - 1$\\ $\st_k(c_{n-1}) + 1$ & if $c(n) > \st_k(c_{n-1} + k - 1$\end{tabular}\right.
\end{equation}
where, again $c_{n-1} \in \C {n-1}$ is the prefix $c(1),\dotsc,c(n-1)$ of $c$.

We will present a specific set of bijections $\majcode k : \S n \to \C n$ using essentially the same technique as in \cite{Rawlings}. In particular, we will show, by construction,
\begin{theorem}\label{bij-thm}
There exists a set of bijections $\majcode k : \S n \to \C n$, such that, for all $k$ and every $\sigma \in \S n$
\begin{enumerate}[(i)]
\item $\majtilde k(\sigma) = \sum(\majcode k(\sigma))$
\item $\des_k(\sigma) = \st_k(\majcode k(\sigma))$
\item $\destilde{k+1}(\sigma) = \st_k(\majcode{k+1}(\sigma))$
\end{enumerate}
\end{theorem}
This immediately gives the following result:
\begin{corollary}\label{pair-corollary}
For any $k$, the bivariate statistics $(\des_k, \majtilde k)$ and $(\destilde{k+1}, \majtilde{k+1})$ have the same distribution.
\end{corollary}
\begin{proof}
By considering the bijection $\majcode k$, we see that $(\des k, \majtilde k)$ has the same distribution as $(\st k, \sum)$. By considering the bijection $\majcode{k+1}$, we see that $(\destilde{k+1}, \majtilde{k+1})$ has the same distribution as $(\st_k, sum)$. Thus, $(\des_k, \majtilde k)$ and $(\destilde {k+1}, \majtilde{k+1})$ have the same distribution.
\end{proof}
For the case $k=1$, Corollary \ref{pair-corollary} gives equation (\ref{desmajdes2maj2}) and thus gives another proof that $\destilde 2$ is Eulerian.

\subsection{The $\majtilde k$-coding scheme}

We will first describe what the mapping $\majcode k$ is, and then prove things about it. The map itself is very easy to state, although it is not initially obvious that it is a bijection, or even that its output will always like in $\C n$.

Define the operation $\cut : \S n \to \S{n-1}$ as follows. For any $\sigma \in \S n$, $\cut(\sigma)$ is the permutation obtained by removing $1$ from $\sigma$ and decrementing all other numbers. For instance, if $\sigma = 4~1~5~2~3$, we remove $1$ to get $4~5~2~3$ and $\cut(\sigma) = 3~4~1~2$.

\begin{mapping}[$\majcode{k}$]\label{mapping}
Take any permutation $\simga \in \S n$. Let $\sigma_n = \sigma$ and $\simga_i = \cut(\sigma_{i+1})$ for $1\le i\le n-1$. Thus we have a sequence $\sigma_n,\dotsc,\sigma(1)$ with $\sigma_i \in \S i$ for all $i$. Now let $c(1) = 0$ and $c(i) = \majtilde k(\sigma_i) - \majtilde k(\sigma_{i-1})$ for $i \ge 2$. Then we let $\majcode k$ be the map which takes $\sigma$ to $c$.
\end{mapping}
Now it is immediate that $\majtilde k(\sigma) = \sum(\majcode k(\sigma))$. Thus, if we can prove that this mapping is a bijection $\S n$ to $\C n$, then we will already satisfy (i) of Theorem \ref{bij-thm}.

To understand this map, we examine what happens when we go from $\cut(\sigma)$ to $\sigma$. Let us define an inverse of the cut operation as follows. For any permutation $\sigma$, let $\ins_i(\sigma)$ be the permutation obtained by incrementing each number of $\sigma$ by $1$ and then inserting $1$ at position $i$. Thus, $\ins_i(\sigma)(i) = 1$ and $\cut(\ins_i(\sigma)) = \sigma$ for any $i$.

We want to study $\majtilde k(\ins_i(\sigma)) - \majtilde k(\sigma)$, for $\sigma \in \S{n-1}$. We claim that as $i$ ranges over $\{1,\dotsc,n\}$, this quantity ranges over $\{0,\dotsc,n-1\}$. Once we prove this, it will be clear that $\majtilde k$ is a surjection $\C n$. Indeed, consider any code $c \in \C n$. Starting with $\sigma_1 = 1$, we can find a sequence $\sigma_1,\dotsc,\sigma_n$ such that $\cut(\sigma_i) = \sigma_{i-1}$ and $\majtilde k(\sigma_i) - \majtilde k(\sigma_{i-1}) = c(i)$. Then $\majcode k(\sigma_n) = c$.

Our analysis of the value of $\majtilde k(\ins_i(\sigma)) - \majtilde k(\sigma)$ relies on a set which we will call $A_k(\sigma)$.
\begin{equation}\label{A-def}
A_k(\sigma) \stackrel{\text{def}}{=} \{1\} \cup \{i > 1 : \sigma(i-1) \ge \sigma(i) + k\text{ or }\sigma(i-1) < k\}
\end{equation}
Why is this set important? It is exacctly the set of positions such that $\des_k$ does not increase when we $\ins$ at that position. More precisely,
\begin{itemize}
\item If $i \in A_k(\sigma)$, then $\des_k(\ins_i(\sigma)) = \des_k(\sigma)$.
\item If $i \not\in A_k(\sigma)$, then $\des_k(\ins_i(\sigma)) = \des_k(\sigma) + 1$.
\end{itemize}
To see why, note that when we insert into $\sigma$ at position $i$, we may gain one $k$-descent if $\sigma(i-1) \ge k$. However, if this is the case, we may also lose one descent (for a net gain of $0$) if $\sigma(i-1) \ge \sigma(i) + k$ was already a $k$-descent.

\begin{example}
Consider the permutation $\sigma = 5~2~1~3~4$ and $k=2$. First increment each number:
\begin{center}\begin{tabular}{ccccccccc}
6 &\phantom{6}& 3 &\phantom{6}& 2 &\phantom{6}& 4 &\phantom{6}& 5
\end{tabular}\end{center}
\emph{Case 1.} $\sigma(i-1) < k$ and $i \in A_k(\sigma)$. For example, take $i = 4$.
\begin{center}\begin{tabular}{ccccccccc}
6 &\phantom{6}& 3 &\phantom{6}& 2 &1& 4 &\phantom{6}& 5
\end{tabular}\end{center}
Here, $\des_k$ does not increase because $2~1$ is not a $2$-descent. Thus, $4 \in A_k(\sigma)$ since $\sigma(3) = 1 < 2$.

\emph{Case 2.} $\sigma(i-1) \ge \sigma(i) + k$ and $i \in A_k(\sigma)$. For example, take $i=2$.
\begin{center}\begin{tabular}{ccccccccc}
6 &1& 3 &\phantom{6}& 2 &\phantom{6}& 4 &\phantom{6}& 5
\end{tabular}\end{center}
Here, $\des_k$ does not increase because, although $6~1$ is indeed a $2$-descent, there was already a $2$-descent, $6~3$ that gets broken by the inserted $1$.

\emph{Case 3.} $i \not\in A_k$. For example, take $i=3$.
\begin{center}\begin{tabular}{ccccccccc}
6 &\phantom{6}& 3 &1& 2 &\phantom{6}& 4 &\phantom{6}& 5
\end{tabular}\end{center}
Here, $\des_k$ does increase because $3~1$ is a $2$-descent. We did not lose any $2$-descents, since $3~2$ was not a $2$-descent.
\end{example}

Now we will see how this set $A_k(\sigma)$ relates to $\majtilde k(\ins_i(\sigma)) - \majtilde k(\sigma)$.

\begin{lemma}\label{sub-lemma}
For $\sigma \in \S{n-1}$, we have
\begin{equation*}
\majtilde k(\ins_i(\sigma)) - \majtilde k(\sigma) = |A_k(\sigma) \cup \{i+1,\dotsc,n\}| + \left\{\begin{tabular}{ll}$0$ & if $i \in A_k(\sigma)$\\ $i-1$ & if $i \not\in A_k(\sigma)$\end{tabular}\right.
\end{equation*}
\end{lemma}
\begin{proof}
Recall from definition (\ref{majkdef}) that
\begin{equation}
\majtilde k(\sigma) = \maj_k(\sigma) + \#\{(i,j) : i < j\text{ and }\sigma(i) < \sigma(j) < \sigma(i) + k\}.
\end{equation}
Let
\begin{equation}
S_k(\sigma) \stackrel{\text{def}}{=} \{(i,j) : i < j\text{ and }\sigma(i) < \sigma(j) < \sigma(i) + k\}
\end{equation}
be the second component of that sum. We consider $\maj_k(\ins_i(\sigma)) - \maj_k(\sigma)$ and $|S_k(\ins_i(\sigma))| - |S_k(\sigma)|$ separately.

We can see that
\begin{equation}\label{eqn-part1}
|S_k(\ins_i(\sigma))| - |S_k(\sigma)| = \#\{j\ge i : \sigma(j) < k\}
\end{equation}
since each such $j$ in the set means we have a pair $(i,j+1)$ for which $i < j=1$ and $\ins_i(\sigma)(i) = 1 < \ins_i(\sigma)(j) < k+1$.

We also have that
\begin{equation}\label{eqn-part2}
\maj(\ins_i(\sigma)) - \maj(\sigma) = \#\{j \ge i : \sigma(j) \ge \simga(j+1) + k\} + \left\{\begin{tabular}{ll}$0$ & if $i \in A_k(\sigma)$\\ $i-1$ & if $i \not\in A_k(\sigma)$\end{tabular}\right.
\end{equation}
because $\#\{j \ge i : \sigma(j) \ge \sigma(j+1) + k\}$ is the number of $k$-descents that get pushed one position to the right, and a $k$-descent is added at position $i-1$ if $i \not \in A_k(\sigma)$. Adding (\ref{eqn-part1}) and (\ref{eqn-part2}) gives the result.
\end{proof}

Given a set $A_k(\sigma)$ for a permutation $\sigma \in \S{n-1}$, we can construct a sequence $\mathcal{A}_k(\sigma) = a_0,\dotsc,a_{n-1}$
\begin{itemize}
\item $a_0 > \dotsb > a_{|A_k(\sigma)|} - 1$ are the elements of $A_k(\sigma)$
\item $a_{|A_k(\sigma)} < \dotsb < a_{n-1}$ are the elements of $\{1,\dotsc,n\} \backslash A_k(\sigma)$, denotes $\overline{A}_k(\sigma)$.
\end{itemize}
\begin{lemma} \label{j-lemma} For $\sigma\in\S{n-1}$ and $0 \le j \le n-1$, we have
\begin{equation*} \majtilde k(\ins_{a_j}(\sigma)) - \majtilde k(\sigma) = j \end{equation*}
\end{lemma}
\begin{proof} By Lemma (\ref{sub-lemma}) it suffices to show that $i = a_j$, then
\begin{equation}
j = |A_k(\sigma) \cup \{i+1,\dotsc,n\}| + \left\{\begin{tabular}{ll}$0$ & if $i \in A_k(\sigma)$\\$i-1$ & if $i \not\in A_k(\sigma)$\end{tabular}\right.
\end{equation}
To see why this is true, we first consider the case that $i \in A_k(\sigma)$. In this case, $i = a_j$ is the $(j+1)^\text{th}$ largest element of $A_k(\sigma)$, so we get
\begin{equation}
j = |A_k(\sigma) \cup \{i+1,\dotsc,n\}|.
\end{equation}
In the other case, suppose $i \not\in A_k(\sigma)$. Since $a_{|A_k(\sigma)|}, \dotsc, a_{n-1}$ are the elements of $\overline{A}_k(\sigma)$, in ascending order, we have $j \ge |A_k(\sigma)|$ and $i = a_j$ is the $(j - |A_k(\sigma) + 1)^\text{th}$ smallest element of $\overline{A}_k(\sigma)$. Hence,
\begin{align}
j - |A_k(\sigma) &= |\overline{A}_k(\sigma) \cup \{1,\dotsc,i-1\}|\\
j - |A_k(\sigma)| &= i - 1 - |A_k(\sigma) \cap \{1,\dotsc,i-1\}|\\
j &= i - 1 + |A_k(\sigma)| - |A_k(\sigma) \cap \{1,\dotsc,i-1\}|\\
j &= i-1+|A_k(\sigma)\cap\{i+1,\dotsc,n\}|
\end{align}
\end{proof}
Lemma \ref{j-lemma} shows that, as we claimed earlier, that $\majtilde k(\ins_i(\sigma)) - \majtilde k(\sigma)$ ranges over $\{0,\dotsc,n-1\}$ as $i$ ranges over $\{1,\dotsc,n\}$. Thus, we have shown that $\majcode k$ satisfies clause (i) of Theorem \ref{bij-thm}.

\subsection{Analysis of $\des_k$ and $\destilde{k+1}$}

We now turn our attention towards proving that $\majcode{k}$ satisfies (ii) and (iii) of Theorem \ref{bij-thm}.

\begin{proof}[Theorem \ref{bij-thm} (ii)] We want to show that $\des_k(\sigma) = \st_k(\majcode{}(\sigma))$, for $\sigma \in \S n$. We will prove this by inducion on $n$. The claim is obvious for $n \le k$, sincce in that case we will always have $\des_k(\sigma) = 0$ and $\st_k(\sigma) = 0$. Now suppose that $n > k$, and suppose that we know that it is true for $n-1$. Let $\sigma \in \S{n-1}$. We have already established that $\des(\ins_i(\sigma)) = \des(\sigma)$ if $i \in A_k(\sigma)$ and $\des(\ins_i(\sigma)) = \des(\sigma) + 1$ otherwise. But we also know from Lemma \ref{j-lemma} that $i \in A_k(\sigma)$ if and only if $\majtilde k(\ins_i(\sigma)) - \majtilde k(\sigma) < |A_k(\sigma)| = \des(\sigma) + k$. This is the same recurrence that $\st_k$ follows, and thus by induction, (ii) holds.
\end{proof}

\begin{proof}[Theorem \ref{bij-thm} (iii)] We want to show that $\destilde{k+1}(\sigma) = \st_k(\majcode{}(\sigma))$, for $\sigma \in \S n$. Once again, we will use induction, noting that $\destilde{k+1}(\sigma) = 0$ for $n \le k$. Assume $\sigma \in \S{n-1}$, with $n > k$.

We have the set $A_k$ which allows us to study $\des_k$. We will define a related set $\widetilde{A}_k$ which will allow us to study $\destilde{k+1}$. For any $\sigma \in \S{n-1}$, let
\begin{equation}\label{Atilde-def}
\widetilde{A}_{k+1}(\sigma) \stackrel{\text{def}}{=} \left\{\begin{tabular}{ll}$A_{k+1}(\sigma)$& if $\sigma(1) > n-k$\\ $A_{k+1}(\sigma)\backslash\{1\}$ & otherwise\end{tabular}\right.
\end{equation}
First, we claim that if $\destilde{k+1}(\ins_i(\sigma)) - \destilde{k+1}(\sigma)$ is $0$ is $i \in \widetilde{A}_{k+1}(\sigma)$, and $1$ otherwise. For $i > 1$, this is true for the same reason that $A_{k+1}(\sigma)$ applies to $\des_{k+1}$. Suppose $i=1$. Then we will get $\ins_i(\sigma)(1) = 1$. Since $n > k$, we will get
\begin{align}
\destilde{k+1}(\ins_i(\sigma)) &= \des_{k+1}(\ins_i(\sigma)) + 1\\
&= \des{k+1}(\sigma) + 1
\end{align}
Furthermore,
\begin{equation*}
\destilde{k+1}(\sigma) = \des{k+1}(\sigma) + 1
\end{equation*}
if $\sigma(1) \le n-k$ and
\begin{equation*}
\destilde{k+1} = \des_{k+1}(\sigma)
\end{equation*}
otherwise. Thus, if $i=1$, $\destilde{k+1}$ increases by $1$ if and only if $1 \in \widetilde{A}_k(\sigma)$.

Our next step is to show that $i \in \widetilde{A}_k(\sigma)$ if and only if
\begin{equation}
\majtilde{k+1}(\ins_i(\sigma)) - \majtilde{k+1}(\sigma) < \destilde{k+1}(\sigma) + k.
\end{equation}
This will complete the proof as it will show $\destilde{k+1}$ follows from the same recurrence as $\st_k$. Thus, we just want to show that $j < \destilde{k+1}(\sigma) + k$ if and only if $i \in \widetilde{A}_{k+1}(\sigma)$. However, it is easy to see that this is the case, as the elements of $\widetilde{A}_k(\sigma)$ are the first $|\widetilde{A}_k(\sigma)| = \destilde{k+1}(\sigma) + k$ elements of the sequence $(a_j)$. This is because teh first $|A_k(\sigma)|$ elements of the sequence are the elements of $A_k(\sigma)$. Thus the result is immediately true if $A_k(\sigma) = \widetilde{A}_k(\sigma)$. It is still true even if the $1$ is missing from $\widetilde{A}_k(\sigma)$, since $1$ is the smallest element of $A_k(\sigma)$,a nd thus appears as the last in the sequence among the elements of $A_k(\sigma)$.
\end{proof}

\section{A quasi-symmetric function proof}

In this section, we present another proof of Proposition \ref{ed} using quasi-symmetric functions suggested by Michelle Wachs. First, recall the definitons in Section 1 for $\des$, $\maj$, $\des_2$, $\maj_2$, $\widetilde{\des_2}$. We also defined $\asc$ and $\ASC$ in Section 1, now similar as for $\des$, we define $\amaj(\sigma)=\sum_{i\in ASC(\sigma)}i$, $\ASC_2(\sigma)=\{i\mid \sigma(i)<\sigma(i+1)-1\}$, $\asc_2(\sigma)=|\ASC_2(\sigma)|$,  $\amaj_2(\sigma)=\sum_{i\in ASC_2(\sigma)}i$, and
 \begin{align*}
\widetilde{\asc_2}(\sigma) &= \left\{\begin{tabular}{ll}$\asc_2(\sigma)$ & if $\sigma(1) = 1$\\ $\asc_2(\sigma) + 1$ & if $\sigma(1) \ne 1$\end{tabular}\right.
\end{align*}
Similar as $\widetilde{\des_2}$, we can see that $\widetilde{\asc_2}$ is also equal distributed as $\cover$ over all permutations of $n$ letters.
The main result of this section is the following. Notice that if we let $q=1$ in the following theorem, we get Proposition \ref{ed}.
\begin{theorem}\label{alg-thm}
For any $n \ge 1$, we have
\begin{equation*}
\sum_{\sigma \in \S n} q^{\amaj_2 (\sigma)} p^{\widetilde{\asc_2}(\sigma)} t^{\des(\sigma^{-1})} =
\sum_{\sigma \in \S n} q^{\maj(\sigma) - \exc(\sigma)} p^{\des(\sigma)} t^{\exc(\sigma)}
\end{equation*}
\end{theorem}
To prove Theorem \ref{alg-thm}, we first state a few results we need about quasi-symmetric functions. Here we will skip the detailed definitions, which can be found in the corresponding references of the following results.
\begin{theorem}[\cite{SW}, (4.8)]\label{SW-result}
For any $n\ge 1$, we have
\begin{equation*}
\sum_{\sigma \in \S n} F_{n,\DEX(\sigma)} t^{\exc(\sigma)} =
\sum_{\sigma\in\S n} F_{n,\DES_2(\sigma)} t^{\des(\sigma^{-1})},
\end{equation*}where $\DEX(\sigma)$ is some statistic related with $\DES$. We will only need the following description of $DEX$.
\end{theorem}
\begin{lemma}[\cite{SW2}, Lemma 2.2]\label{dex}
$$\sum_{i\in\DEX(\sigma)}i=\maj(\sigma)-\exc(\sigma),$$
 \begin{align*}
|\DEX(\sigma)| &= \left\{\begin{tabular}{ll}$\des(\sigma)$ & if $\sigma(1) = 1$\\ $\des(\sigma) - 1$ & if $\sigma(1) \ne 1$\end{tabular}\right.
\end{align*}
\end{lemma}

\begin{proof}[Proof of Theorem \ref{alg-thm}]
By palindromicity,  we have
\begin{equation*}
\sum_{\sigma\in\S n} F_{n,\DES_2(\sigma)} t^{\des(\sigma^{-1})} = \sum_{\sigma\in\S n} F_{n,\ASC_2(\sigma)}t^{\des(\sigma^{-1})}
\end{equation*}
Thus by Theorem \ref{SW-result} we have
\begin{equation*}
\sum_{\sigma \in \S n} F_{n,\DEX(\sigma)}t^{\exc(\sigma)} = \sum_{\sigma \in \S n} F_{n,\ASC_2(\sigma)}t^{\des(\sigma^{-1})}.
\end{equation*}
Now apply specializations to the above. By Lemma 5.2 of \cite{Gessel} (see Lemma 2.1 of \cite{SW2}), we get
\begin{equation*}
\sum_{\sigma \in \S n} q^{\Sigma\DEX(\sigma)}p^{|\DEX(\sigma)|}t^{\exc(\sigma)} =
\sum_{\sigma \in \S n} q^{\Sigma\ASC_2(\sigma)} p^{\asc_2(\sigma)} t^{\des(\sigma^{-1})}.
\end{equation*}
Apply Lemma \ref{dex}.
\begin{align*}
\sum_{\sigma \in \S n : \sigma(1) = 1} q^{\maj(\sigma) - \exc(\sigma)} p^{\des(\sigma)} t^{\exc(\sigma)} + \sum_{\sigma\in\S n : \sigma(1) \ne 1} q^{\maj(\sigma) - \exc(\sigma)} p^{\des(\sigma) - 1} t^{\exc(\sigma)} =\\
\sum_{\sigma\in\S n : \sigma(1) = 1} q^{\amaj_2(\sigma)} p^{\widetilde{\asc_2}(\sigma)} t^{\des(\sigma^{-1})} + \sum_{\sigma \in \S n : \sigma(1) \ne 1} q^{\amaj_2(\sigma)} p^{\widetilde{\asc_2}(\sigma) - 1}t^{\des(\sigma^{-1})}
\end{align*}
Hence we just need to show that
\begin{equation*}
\sum_{\sigma\in\S n : \sigma(1) = 1} q^{\maj(\sigma) - \exc(\sigma)} p^{\des(\sigma)} t^{\exc(\sigma)} = \sum_{\sigma \in \S n : \sigma(1) = 1} q^{\amaj_2(\sigma)} p^{\widetilde{\asc_2}(\sigma)} t^{\des(\sigma^{-1})}.
\end{equation*}
If we assume by induction that this is true for $n-1$, then we have
\begin{align*}
\sum_{\sigma \in \S n : \sigma(1) = 1} q^{\maj(\sigma) - \exc(\sigma)} p^{\des(\sigma)} t^{\exc(\sigma)}
&= \sum_{\sigma\in\S{n-1}} q^{\maj(\sigma) - \exc(\sigma)}(pq)^{\des(\sigma)}t^{\exc(\sigma)}\\
&= \sum_{\simga\in\S{n-1}} q^{\amaj_2(\sigma)}(pq)^{\widetilde{\asc_2}(\sigma)}t^{\des(\sigma^{-1})}
\end{align*}
But then
\begin{align*}
&\sum_{\sigma\in\S{n-1}}q^{\amaj_2(\sigma)}(pq)^{\widetilde{\asc_2}(\sigma)}t^{\des(\sigma^{-1})}\\
=& \sum_{\sigma\in\S{n-1} : \sigma(1) = 1} q^{\amaj_2(\sigma)} (pq)^{\asc_2(\sigma)} t^{\des(\sigma^{-1})}
+ \sum_{\sigma\in\S{n-1} : \sigma(1)\ne 1} q^{\amaj_2(\sigma)}(pq)^{\asc_2(\sigma)+1}t^{\des(\sigma^{-1})}\\
=& \sum_{\sigma \in \S n : \sigma(1) = 1, \sigma(2) = 2} q^{\amaj_2(\sigma)} p^{\asc_2(\sigma)} t^{\des(\sigma^{-1})}
+ \sum_{\sigma\in\S n : \sigma(1) = 1, \sigma(2) \ne 2} q^{\amaj_2(\sigma)} p^{\asc_2(\sigma)} t^{\des(\sigma^{-1})}\\
=& \sum_{\sigma\in\S n, \sigma(1) = 1} q^{\amaj_2(\sigma)} p^{\widetilde{\asc_2}(\sigma)} t^{\des(\sigma^{-1})}.
\end{align*}
So we are done.
\end{proof}

\section{Acknowledgements}
We thank Richard Stanley and Michelle Wachs for
helpful discussions. We also thank MIT SPUR program (the Summer Program in Undergraduate Research of the MIT Mathematics Department).

\bigskip

\filbreak \noindent Travis Hance\\
Department of Mathematics\\
Massachusetts Institute of Technology\\
Cambridge, MA 02139\\
{\tt tjhance@MIT.EDU}
\bigskip

\filbreak\noindent Nan Li\\
Department of Mathematics\\
Massachusetts Institute of Technology\\
Cambridge, MA 02139\\
{\tt nan@math.mit.edu}

\begin{thebibliography}{99}

\bibitem{Clarke} \textsc{Clarke, R., Steingr\'imsson, E., and Zeng, J.} \emph{New Euler-Mahonian permutation statistics}, Adv. in Appl. Math \textbf{18} (1997) 237-270.

\bibitem{Foata2} \textsc{Foata, D.} \emph{Distribution Eul\'riennes et Mahoniennes sur le groupe des permutations}, in M. Aigner (ed.), Higher Combinatorics, 27-49, D. Reidel, Boston, Berlin Combinatorics Symposium, 1976.

\bibitem{Foata3} \textsc{Foata, D.} Rearrangements of words, in \emph{M. Lothaire, Combinatorics on
Words}, (ed.) G.-C. Rota, Vol. 17, Encyclopedia of Math. and its Appl., Addison-Wesley Publishing Company, 1983.


\bibitem{Foata5} \textsc{Foata, D. and Sch\"utzenberger, M.-P.} \emph{Major index and inversion number of permutations}, Math. Nachr. \textbf{83} (1978), 143-159.

\bibitem{Foata} \textsc{Foata, D. and Zeilberger, D.} \emph{Denert's permutation statistic is indeed Euler-Mahonian}. Studies in Appl. Math. \textbf{83} (1990), 31-59.

\bibitem{Gessel} \textsc{Gessel, I. and Reutenauer, C.} \emph{Counting permutations with given cycle
structure and descent set.} J. Combin. Theory Ser. A 64 (1993), 189-215.


\bibitem{ln} \textsc{Li, N.} \emph{Ehrhart $h^*$-vectors of hypersimplices}, arXiv:1104.5292, to appear in Discrete and Computational Geometry.

\bibitem{MacMahon} \textsc{MacMahon, P.} \emph{Combinatory Analysis}, vols. 1 and 2. Cambridge Univ.
Press, Cambridge, 1915 (reprinted by Chelsea, New York, 1955).

\bibitem{Rawlings2} \textsc{Rawlings, D.} \emph{Permutation and multipermutation statistics}, Europ. J. Combinatorics, \textbf{2} (1981), 67-78.

\bibitem{Rawlings} \textsc{Rawlings, D.} \emph{The r-major index}. J. Combin. Theory, Ser. A 31 (1981) 175-183.

\bibitem{Simion} \textsc{Simion, S. and Stanton, D.} \emph{Specializations of generalized Laguerre polynomials}, SIAM J. Math. Anal. \textbf{25} (1994), 712-719

\bibitem{SW} \textsc{Shareshian, J. and Wachs, M.} \emph{Chromatic quasisymmetric functions and
Hessenberg varieties}, arXiv:1106.4287, to appear in the Proceedings of De Giorgi Center Program on Configuration Spaces.

\bibitem{SW2} \textsc{Shareshian, J. and Wachs, M.} \emph{Eulerian quasisymmetric functions.} Adv.
Math. 225 (2010), no. 6, 2921-2966.

\bibitem{Skandera} \textsc{Skandera, M.} \emph{An Eulerian partner for inversions}. S\'eminaire Lotharingien de Combinatoire 46 (2001), Article B46d.

\bibitem{Stanley} \textsc{Stanley, R.} \emph{Binomial posets, M\"obius inversion and permutation enumeration}, J. Comb. Theory, \textbf{A, 20} (1976), 712-719.

\bibitem{EC1} \textsc{Stanley, R.} What is Enumerative Combinatorics?, In \emph{Enumerative Combinatorics} (9-114), vol. 1, ed. 2, Cambridge University Press, 2011.


\end{thebibliography}
\end{document}